\documentclass[nopreprintline,12pt]{elsarticle}

\usepackage{amssymb}
\usepackage[numbers]{natbib}
\biboptions{sort&compress}
\usepackage{url}
\usepackage{amsmath}
\usepackage{amsthm}
\newtheorem{lemma}{Lemma}[section]
\newtheorem{theorem}{Theorem}[section]
\newtheorem{problem}{Problem}
\newtheorem{corollary}{Corollary}[section]
\newtheorem{conjecture}{Conjecture}
\newtheorem{observation}{Observation}

\journal{Discrete Applied Mathematics}

\begin{document}

\begin{frontmatter}



\title{Recent studies on the super edge-magic deficiency of graphs}


\author[label1]{Rikio Ichishima}
\affiliation[label1]{organization={Department of Sport and Physical Education, Faculty of Physical Education, Kokushikan University},
  addressline={7-3-1 Nagayama},
  city={Tama-shi},
  postcode={206-8515},
  state={Tokyo},
  country={Japan}}

\author[label2]{S.C.  L\'{o}pez}
\affiliation[label2]{organization={Departament de Matem\`{a}tica, Universitat de Lleida},
  addressline={C/Jaume II, 69, 25001},
  city={Lleida},
%  state={},
  country={Spain}}
\author[label3]{Francesc A. Muntaner-Batle}
\affiliation[label3]{organization={Graph Theory and Applications Research Group, School of Electrical Engineering and Computer Science, Faculty of Engineering and Built Environment, The University of Newcastle},
  addressline={NSW 2308},
%  city={},
%  postcode={},
%  state={},
  country={Australia}}
\author[label4]{Yukio Takahashi}
\affiliation[label4]{organization={Department of Science and Engineering, Faculty of Electronics and Informations, Kokushikan University},
  addressline={4-28-1 Setagaya},
  city={Setagaya-ku},
  postcode={154-8515},
  state={Tokyo},
  country={Japan}}

\begin{abstract}
A graph $G$ is called edge-magic if there exists a bijective function $f:V\left(G\right) \cup E\left(G\right)\rightarrow \left\{1, 2, \ldots , \left\vert V\left( G\right) \right\vert +\left\vert E\left( G\right) \right\vert \right\}$ such that $f\left(u\right) + f\left(v\right) + f\left(uv\right)$ is a constant for each $uv\in E\left( G\right) $. 
Also, $G$ is said to be super edge-magic if $f\left(V \left(G\right)\right) =\left\{1, 2, \ldots , \left\vert V\left( G\right) \right\vert \right\}$. 
Furthermore, the super edge-magic deficiency $ \mu_{s}\left(G\right)$ of a graph $G$ is defined to be either the smallest nonnegative integer $n$ with the property that $G \cup nK_{1}$ is super edge-magic or $+ \infty$ if there exists no such integer $n$. 
In this paper, we introduce the parameter $l\left(n\right)$ as the minimum size of a graph $G$ of order $n$ for which all graphs of order $n$ and size at least $l\left(n\right)$ have $\mu_{s} \left( G \right)=+\infty $, 
and provide lower and upper bounds for $l\left(G\right)$. 
Imran, Baig, and Fe\u{n}ov\u{c}\'{i}kov\'{a} established that for integers $n$ with $n\equiv 0\pmod{4}$, $ \mu_{s}\left(D_{n}\right) \leq 3n/2-1$, where $D_{n}$ is the cartesian product of the cycle $C_{n}$ of order $n$ and the complete graph $K_{2}$ of order $2$.
We improve this bound by showing that $ \mu_{s}\left(D_{n}\right) \leq n+1$ when $n \geq 4$ is even.
Enomoto, Llad\'{o}, Nakamigawa, and Ringel posed the conjecture that every nontrivial tree is super edge-magic. 
We propose a new approach to attak this conjecture. 
This approach may also help to resolve another labeling conjecture on trees by Graham and Sloane. 

\end{abstract}




\end{frontmatter}


%\section{}
%\label{}
%%%%%%%%%%%%%%%%%%%%%%%%%%%%%%%%%%%%%%%%%%%%%%%%%%

\section{Introduction}
Unless stated otherwise, the graph-theoretical notation and terminology used here will follow Chartrand and Lesniak \cite{CL}. 
In particular, the \emph{vertex set} of a graph $G$ is denoted by $V \left(G\right)$, while the \emph{edge set} of $G$ is denoted by $E\left (G\right)$. The \emph{cycle} of order $n$ and the \emph{complete graph} of order $n$ are denoted by $C_{n}$ and $K_{n}$, respectively.

For the sake of brevity, we will use the notation $\left[ a, b\right] $ for the interval of integers $x $ such that $a\leq x\leq b$.
Kotzig and Rosa \cite{KR} initiated the study of what they called magic valuations. 
This concept was later named edge-magic labelings by Ringel and Llad\'{o} \cite{RL} and this has become the popular term.
A graph $G$ is called \emph{edge-magic} if there exists a bijective function $f:V\left(G\right) \cup E\left(G\right)\rightarrow \left[1, \left\vert V\left( G\right) \right\vert +\left\vert E\left( G\right) \right\vert \right]$ such that $f\left(u\right) + f\left(v\right) + f\left(uv\right)$ is a constant for each $uv\in E\left( G\right) $. Such a function is called an \emph{edge-magic labeling}. More recently, they have also been referred to as edge-magic total labelings by Wallis \cite{Wallis}.

Enomoto, Llad\'{o}, Nakamigawa, and Ringel \cite{ELNR} introduced a particular type of edge-magic labelings, namely, super edge-magic labelings. 
They defined an edge-magic labeling of a graph $G$ with the additional property that 
$f\left(V \left(G\right)\right)=\left[1, \left\vert V\left( G\right) \right\vert \right]$ to be a \emph{super edge-magic labeling}. 
Thus, a \emph{super edge-magic graph} is a graph that admits a super edge-magic labeling. 

Lately, super edge-magic labelings and super edge-magic graphs are called by Wallis \cite{Wallis} strong edge-magic total labelings and strongly edge-magic graphs, respectively. According to the latest version of the survey on graph labelings by Gallian \cite{Gallian} available to the authors, Hegde and Shetty \cite{HS2} showed that the concepts of super edge-magic graphs and strongly indexable graphs (see \cite{AH} for the definition of a strongly indexable graph) are equivalent.

The following result found in \cite{FIM} provides necessary and sufficent conditions for a graph to be super edge-magic, which will prove later to be useful.

\begin{lemma}
\label{trivial}
A graph $G$ is super edge-magic if and only if
there exists a bijective function $f:V\left( G\right) \rightarrow \left[ 1, \left\vert V\left( G\right) \right\vert \right] $ such that the set 
\begin{equation*}
S=\left\{ f\left( u\right) +f\left( v\right) \left| uv\in E\left( G\right)
\right. \right\}
\end{equation*}%
consists of $\left\vert E\left( G\right) \right\vert $ consecutive integers. In such a case, $f$ extends to a super
edge-magic labeling of $G$ with magic constant $k=\left\vert V\left( G\right) \right\vert +\left\vert E\left( G\right) \right\vert +s$, where $s=\min
\left( S\right) $ and 
\begin{equation*}
S=\left[ k-\left( \left\vert V\left( G\right) \right\vert +\left\vert E\left( G\right) \right\vert \right) ,k-\left( \left\vert V\left( G\right) \right\vert +1\right) \right] \text{.}
\end{equation*}
\end{lemma}

Enomoto, Llad\'{o}, Nakamigawa, and Ringel \cite{ELNR} showed that caterpillars are super edge-magic and posed the following conjecture. 

\begin{conjecture}
\label{super_edge-magic_tree_conjecture}
Every nontrivial tree is super edge-magic.
\end{conjecture}

Lee and Shan \cite{LS} have verified the above conjecture for trees with up to $17$ vertices with a computer. 
Fukuchi and Oshima \cite{FO} have shown that if $T$ is a tree of order $n \geq 2$ such that $T$ has diameter greater than or equal to $n-5$, then $T$ is super edge-magic. 
Various classes of banana trees (see \cite{Gallian} for the definition) that have super edge-magic labelings have been found independently by Swaminathan and Jeyanthi \cite{SJ}, and Hussain, Baskoro, and Slamin \cite{HBS}. 
Fukuchi \cite{Fukuchi} showed how to recursively create super edge-magic trees from certain kinds of existing super edge-magic trees. 
Ngurah, Baskoro, and Simanjuntak \cite{NBS} provided a method for constructing new (super) edge-magic graphs from existing ones. 
For further knowledge on the progress of Conjecture \ref{super_edge-magic_tree_conjecture}, the authors suggest that the reader consult the extensive survey by Gallian \cite{Gallian}.

For every graph $G$, Kotzig and Rosa \cite{KR} proved that there exists an edge-magic graph $H$ such that $H= G \cup nK_{1}$ for some nonnegative integer $n$. 
This motivated them to define the edge-magic deficiency of a graph. 
The \emph{edge-magic deficiency} $\mu \left(G\right)$ of a graph $G$ is the smallest nonnegative integer $n$ for which $G \cup nK_{1}$ is edge-magic. Inspired by Kotzig and Rosa's notion, the concept of \emph{super edge-magic deficiency} $\mu_{s} \left(G\right)$ of a graph $G$ was analogously defined in \cite{FIM2} as either the smallest nonnegative integer $n$ with the property that $G \cup nK_{1}$ is super edge-magic or $+ \infty$ if there exists no such integer $n$.  Thus, the super edge-magic deficiency of a graph $G$ is a measure of how “close” (“ far ”) $G$ is to (from) being super edge-magic. 

An alternative term exists for the super edge-magic deficiency, namely, the vertex dependent characteristic. 
This term was coined by Hedge and Shetty \cite{HS}. 
In \cite{HS}, they gave a construction of polygons having same angles and distinct sides using the result on the super edge-magic deficiency of cycles provided in \cite{FIM3}.

\section{Lower and upper bounds}
It is known from \cite{FIM3} that $\mu_{s} \left(K_{n}\right)=+\infty$ for every integer $n \geq5$. 
It follows that for every integer $n$ with $n \neq 1, 2, 3, 4$, 
there exists a positive integer $l\left(n\right)$ with the property that if $G$ is a graph of order $n$ and size at least $l\left(n\right)$, 
then $\mu_{s} \left(G\right)=+\infty$.
It is interesting to determine the exact value of $l\left(n\right)$. 
However, it seems that this is a very hard problem. 
In this section, we present lower and upper bounds for this value.

We begin with the following lower bound for $l\left(n\right)$. 

\begin{theorem}
\label{lower_bound}
For every integer $n \geq4$,
\begin{equation*}
l\left(n\right) \geq \lceil n/2\rceil \left(\lfloor n/2\rfloor +1\right)+1 \text{.}
\end{equation*}
\end{theorem}

\begin{proof}
Define the graph $G$ with 
\begin{equation*}
V\left(G\right) = \left\{x_i \left| \right. i\in \left[1, \lceil n/2\rceil \right] \right\} \cup \left\{y_i \left| \right. i\in \left[1, \lfloor n/2\rfloor \right] \right\}
\end{equation*}
and

\begin{eqnarray*}
  E\left(G\right) & = & \left\{x_{i}x_{j} \left|i \right. \in \left[1, \lceil n/2\rceil\right] \text{ and } j \in \left[1, \lfloor n/2\rfloor\right]\right\} \\
  & & \cup \left\{x_{1}x_{i} \left| \right. i \in\left[2, \lceil n/2\rceil\right]\right\} \cup \left\{y_1y_{\lfloor n/2\rfloor}\right\} \text{.}
\end{eqnarray*}

%\begin{equation*}
% E\left(G\right) = \left\{x_{i}x_{j} \left|i \right. \in \left[1, \lceil n/2\rceil\right] \text{ and } j \in \left[1, \lfloor n/2\rfloor\right]\right\} \cup 
%\left\{x_{1}x_{i} \left| \right. i \in\left[2, \lceil n/2\rceil\right]\right\} \cup \left\{y_1y_{\lfloor n/2\rfloor}\right\} \text{.}
%\end{equation*}

Now, consider the vertex labeling 
$f:V\left(G\right)\rightarrow \left[1, \lceil n/2\rceil \lfloor n/2\rfloor +1\right]$
such that
\begin{equation*}
f\left(v\right)=\left\{ 
\begin{tabular}{ll}
$i $ & if $v=x_{i}$ and $i\in \left[1, \lceil n/2 \rceil\right]$\\ 
$\lceil n/2 \rceil i+1$ & if $v=y_{i}$ and $i\in \left[1, \lfloor n/2 \rfloor\right]$
\end{tabular}%
\right.
\end{equation*}
Then
\begin{equation*}
\begin{tabular}{ll}
  $\left\{ f\left(x_{1}\right)+ f\left(x_{i}\right) \left| \right. i \in \left[2, \lceil n/2 \rceil\right\}=\left[3, \lceil n/2 \rceil+1\right] \right. \text{,}$ \\
  $\left\{ f\left(x_{i}\right)+ f\left(y_{j}\right) \left| \right. i \in \left[1, \lceil n/2\rceil\right] \right.  \text{ and } j \in \left[1, \lfloor n/2 \rfloor  \right] = \left[\lceil n/2 \rceil+2, \lceil n/2 \rceil \left(\lfloor n/2 \rfloor +1 \right)+1\right] \text{,}$ \\
  $\left\{ f\left(y_{1}\right)+ f\left(y_{\lfloor n/2 \rfloor}\right) \right\}=\left\{\lceil n/2 \rceil\left(\lfloor n/2 \rfloor+1\right)+2 \right\}$\text{.}
\end{tabular}
\end{equation*}
Since $\left\vert E\left( G\right) \right\vert=\lceil n/2 \rceil \left(\lfloor n/2 \rfloor+1\right)$, it follows that the set

\begin{equation*}
S=\left\{f\left(x\right)+f\left(y\right) \left| \right. xy \in E\left(G\right)\right\}
\end{equation*}
is a set of $\left\vert E\left( G\right) \right\vert$ consective integers.
This shows that $\mu_{s}\left(G\right)<+\infty$.
Hence, there exists a graph $G$ of order $n$ and size $\lceil n/2\rceil \left(\lfloor n/2\rfloor +1\right)$ so that $\mu_{s}\left(G\right)<+\infty$.
Therefore, $l\left(n\right) \geq \lceil n/2\rceil \left(\lfloor n/2\rfloor +1\right)+1$.
\end{proof}

For a finite set $S$ of integers, we define the \emph{gap} $\Gamma \left(S\right)$ of $S$ to be

\begin{equation*}
\Gamma \left(S\right)=\left(\max(S)-\min(S)+1\right)-\left\vert S \right\vert \text{.}
\end{equation*}
Then the following fact is a consequence of the above defition.

\begin{observation}
\label{obervation_gap} 
Let $S$ be a finite set of integers. 
Then $S$ is a set of consecutive integers if and only if $\Gamma\left(S\right)=0$.
\end{observation}

To study graphs for which the clique number $\omega\left(G\right)$ of a graph $G$ (the largest order among the complete subgraph of $G$) is large in relation to the size of the graph, we have resorted to the theory of well spread sets introduced by Kotzig \cite{ Kotzig}.
A set $\left\{x_{i} \left| \right. i\in \left[1, n\right] \right\} \subset \mathbb{N}$ with $x_{1}<x_{2}<\cdots<x_{n}$ is a
%\emph{well spread set} or a \emph{weak Sidon set} as referred by Ruzsa \cite{ Ruzsa} (WS-set for short) if
\emph{well spread set} or \emph{a weak Sidon set} (WS-set for short)  by Ruzsa \cite{Ruzsa} if
the sums $x_{i}+x_{j}$ ($i<j$) are all different. 
Furthermore, we define the \emph{smallest span of pairwise sums} $\rho^{*} \left(n\right)$ of cardinality of $n$ to be 
\begin{equation*}
\rho^{*} \left(n\right)=\min\left\{x_{n}+x_{n-1}-x_{2}-x_{1}+1 \left| \right. \{x_{1}<x_{2}<\cdots<x_{n} \} \text{ is WS-set} \right\} \text{.} 
\end{equation*}

The following lemma found by Kotzig \cite{Kotzig} provides a lower bound of $\rho^{*} \left(n\right)$ for every integer $n \geq7$.

\begin{lemma}
\label{Kotzig}
For every integer $n \geq7$, 
\begin{equation*}
\rho^{*} \left(n\right) \geq n^{2}-5n+14 \text{.}
\end{equation*}
\end{lemma}

With the aid of Lemma \ref{Kotzig}, it is possible to present the following result.
\begin{theorem}
\label{Kn+1-e}
For every integer $n \geq 7$, 
\begin{equation*}
\mu_{s} \left(K_{n+1}-e\right)=+\infty \text{,}
\end{equation*}
where $e \in E\left(K_{n+1}\right)$.
\end{theorem}

\begin{proof}
Assume, to the contrary, that $\mu_{s} \left(K_{n+1}-e\right)=k$ for some positive integer $k$.
Then there exists a bijective function $f:V\left(\left(K_{n+1}-e\right) \cup kK_{1}\right) \rightarrow \left[1, n+1+k \right]$  such that the set

\begin{equation*}
S=\left\{ f\left(x\right)+f\left(y\right) \left|xy \in E\left(\left(K_{n+1}-e\right) \cup kK_{1}\right\} \right. \right.
\end{equation*}
is a set of $\tbinom{n+1}{2}-1$ consective integers, that is, 
$\Gamma\left(S\right)=0$ by Obsrvation \ref{obervation_gap}.

Now, assume that $u, v \in V\left(K_{n+1}-e\right)$ but $uv \notin E\left(K_{n+1}-e\right)$.
Also, consider the subgraph of $K_{n+1}-e$ obtained by elminating vertex $u$ so that
the resulting subgraph is $K_{n}$ and consider the set 

\begin{equation*}
S^{\prime}=\left\{ f\left(x\right)+f\left(y\right) \left|xy \in E\left(\left(K_{n+1}-e\right) \cup kK_{1}\right)\backslash \left\{u\right\}\right\} \text{.} \right. 
\end{equation*}
Since the sums considered in $S^{\prime}$ are the same sums as the sums considered in $S$ but for $\left(n-1 \right)$ sums (the ones corresponding to edges incident with $u$), it follows that $\Gamma \left(S^{\prime}\right) \leq n-1$.
On the other hand, the set 

\begin{equation*}
\Omega=\left\{ f\left(x\right) \left|x \in V\left(\left(K_{n+1}-e\right)\backslash \left\{u\right\}\right\} \right. \right.
\end{equation*}
is a well spread set of cardinality $n$. 
It follows from Lemma \ref{Kotzig} that
\begin{equation*}
\max(W) -\min(W)+1 \geq n^{2}-5n+14 \text{,}
\end{equation*}
where $W=\left\{ f\left(x\right)+ f\left(y\right) \left|f\left(x\right), f\left(y\right) \in \Omega \text{ and } f\left(x\right) \neq f\left(y\right)\right\} \right.$.
This implies that 
\begin{equation*}
\Gamma\left(S^{\prime}\right) \geq n^{2}-5n+14-\dbinom{n}{2}=\left(n^{2}-9n+28\right)/2 \text{.}
\end{equation*}
Therefore, 
\begin{equation*}
\left(n^{2}-9n+28\right)/2 \leq \Gamma\left(S^{\prime}\right) \leq n-1
\end{equation*}
for all integers $n \geq7$.
However, since $\left(n^{2}-9n+28\right)/2>n-1$ for all integers $n \geq7$, 
it follows that $\Gamma\left(S^{\prime}\right)>n-1$, producing a contradiction.
\end{proof}

In fact, for any integer $n \geq 8$, the preceding result provides us with an upper bound on $l\left(n\right)$, 
since for these values of $n$, we know that $\mu_{s}\left(K_{n}\right)= \mu_{s}\left(K_{n}-e\right)=+\infty$, 
where $e \in E\left(K_{n}\right)$. 
Therefore, we have the following upper bound for $l\left(n\right)$.

\begin{corollary}
\label{upper_bound}
For every integer $n \geq 7$,
\begin{equation*}
l\left(n\right) \leq \dbinom{n}{2}-2 \text{.}
\end{equation*}
\end{corollary}

From now on, let $K_{n}-\alpha e$ denote the set of all graphs obtained from $K_{n}$ by removing exactly $\alpha$ edges,
where $\alpha$ is a positive integer. 
Our next theorem generalizes the preceding result.

\begin{theorem}
%\label{Kn+1-e}
For a fixed positive integer $\alpha$, 
there exists some positive integer $j\left(\alpha\right)$ such that if $n>j\left(\alpha\right)$, 
then $\mu_{s}\left(G\right)=+\infty$ for all $G \in K_{n}-\alpha e$, 
where $n>2\alpha$.
\end{theorem}

\begin{proof}
For a fixed positive integer $\alpha$, assume that $n>2\alpha$, where $n$ is a positive integer.
Let $G \in K_{n}-\alpha e$ and suppose, to the contrary, that for every integer $n \in \mathbb{N}$,
there exists some $G \in K_{n}-\alpha e$ such that $\mu_{s}\left(G\right)<+\infty$.
Then there exists an injective function $f:V\left(G\right) \rightarrow \mathbb{N}$ such that the set
\begin{equation*}
S=\left\{f\left(x\right)+f\left(y\right) \left|xy \in E\left(G\right)\right\} \right.
\end{equation*}
is a set of $\left\vert E\left( G\right) \right\vert$ consective integers.
Also, notice that there are at most $2\alpha$ vertices that have degree at most $n-2$, 
since there are exactly $\alpha$ edges missing to form $K_{n}$.
So, if we elminate all vertices of degree at most $\left(n-2\right)$, 
then the resulting graph is a complete graph.
If it is $K_{n-2\alpha}$, then we are done; otherwise, keep elminating vertices until we arrive at $K_{n-2\alpha}$.

Now, consider the set $S^{\prime}=\left\{f\left(x\right)+f\left(y\right) \left|xy  \right. \in E\left(K_{n-2\alpha})\right\} \right.$.
Since  $S^{\prime}$ comes from the set $S$ by removing at most $2\alpha \left(n-1\right)$ sums induced by edges,
it follows that 
\begin{equation*}
\Gamma\left(S^{\prime}\right) \leq 2\alpha \left(n-1\right) \text{.}
\end{equation*}
On the other hand, it follows from Lemma \ref{Kotzig} that 
\begin{equation*}
\Gamma\left(S^{\prime}\right) \geq \left(n-2\alpha\right)^{2}-5\left(n-2\alpha\right)+14-\dbinom{n-2\alpha}{2} \text{.}
\end{equation*}
This together with the preceding inequality implies that
\begin{equation*}
\left(n-2\alpha\right)^{2}-5\left(n-2\alpha\right)+14-\dbinom{n-2\alpha}{2} \leq \Gamma\left(S^{\prime}\right) \leq 2\alpha \left(n-1\right) \text{.}
\end{equation*}
However, since the inequality
\begin{equation*}
\left(n-2\alpha\right)^{2}-5\left(n-2\alpha\right)+14-\dbinom{n-2\alpha}{2}>2\alpha \left(n-1\right)
\end{equation*}
is valid for all integers 
\begin{equation*}
n>\dfrac{\left(8\alpha+9\right)+\sqrt{\left(8\alpha+9\right)^{2}-\left(16\alpha^{2}+88\alpha+112\right)}}{2} \text{,}
\end{equation*}
this produces a contradiction.
\end{proof}

Observe that if we let $\alpha =2$ and we compute the minimum value of $n$ that satisfies the previous inequality, 
we get $n \geq 21$. 
This means that for $n \geq 21$, we have $\mu_{s}\left(G\right)=+\infty$ for any graph $G \in K_{n}-2e$. 
However, it is also known that $\mu_{s}\left(K_{n}\right)=\mu_{s}\left(K_{n}-e\right)= +\infty$. 
Therefore, $l\left(n\right) \leq \tbinom{n}{2}-2$ for $n \geq 21$. 
Continue in this manner, we can obtain upper bounds on $l\left(n\right)$ for sufficiently large integers $n$ as the next result indicates.

\begin{corollary}
\label{upper_bound_general}
For sufficiently large integers $n$,
\begin{equation*}
l\left(n\right) \leq \dbinom{n}{2}-\alpha \text{,}
\end{equation*}
where $\alpha$ is a fixed positive integer such that $n>2\alpha$.
\end{corollary}

\section{An improved upper bound}

The \emph{prism} $D_{n}$ is defined to be the cartesian product of $C_{n}$ and $K_{2}$. 
The prism is also known to be the Cayley graph of the dihedral group $D_{n}$ with respect to the generating set $\left\{x, x^{-1}, y\right\}$.
It was proved in \cite {FIM} that if $n \geq 3$ is odd, then $D_{n}$ is super edge-magic, that is, $\mu_{s}\left(D_{n}\right)=0$ in this case. 
Ngurah and Baskoro \cite{NB} showed that if $n \geq4$ is even, then $D_{n}$ is not edge-magic, implying that $\mu_{s}\left(D_{n}\right)\geq \mu\left(D_{n}\right) \geq 1$ by definitions.
Imran, Baig, and Fe\u{n}ov\u{c}\'{i}kov\'{a} \cite{IBF} established the following upper bound for $\mu_{s}\left(D_{n}\right)$.
 
\begin{theorem}
\label{old_bound}
For integers $n$ with $n\equiv 0\pmod{4}$, 
\begin{equation*}
\mu_{s}\left(D_{n}\right) \leq 3n/2-1 \text{.}
\end{equation*}
\end{theorem}

In this section, we provide an improved upper bound for $\mu_{s}\left(D_{n}\right)$ when $n \geq 4$ is even,
and propose an open problem for $\mu_{s}\left(D_{n}\right)$ when $n \geq 6$ is even.
To proceed, we introduce some additional definitions and results next.

The graph labeling method that has received the most attention over the years was originated with a paper by Rosa \cite{Rosa} who called them $\beta$-valuations. 
 A few years later, Golomb \cite{Golomb} called these labelings graceful and this is the term that has been used since then. For a graph $G$, an injective function $f:V\left(G\right) \rightarrow \left[1, \left\vert E\left( G\right) \right\vert \right]$  is called a  \emph{graceful labeling} if each $uv\in E\left( G\right) $ is labeled $\left\vert f\left(u\right)-f\left(v\right) \right\vert$ and the resulting edge labels are distinct. 
Rosa \cite{Rosa} also introduced the concept of $\alpha$-valuations (a particular type of graceful labelings) as a tool for decomposing the complete graph into isomorphic subgraphs. 
A graceful labeling $f$ is called an $\alpha \emph{-valuation}$ if there exists an integer $\lambda$ so that 

\begin{equation*}
  \min \left\{f\left(u\right), f\left(v\right)\right\} \leq \lambda < \max \left\{f\left(u\right), f\left(v\right) \right\}
\end{equation*}  
for each $uv\in E\left( G\right) \text{.}$  

Douglas and Reid \cite{DR} obtained the following result.

\begin{theorem}
For every integer $n \geq 2$, the prism $D_{2n}$ has an $\alpha$-valuation.
\end{theorem}

The following result found in \cite{IO} shows how $\alpha$-valuations are useful for computing the super edge-magic deficiency of certain graphs.

\begin{theorem}
Let $G$ be a graph without isolated vertices that has an $\alpha$-valuation. Then 
\begin{equation*}
\mu_{s}\left (G\right) \leq \left\vert E\left( G\right) \right\vert-\left\vert V\left( G\right) \right\vert + 1 \text{.}
\end{equation*}
\end{theorem}

For every integer $n \geq 3$, we have $\left\vert V\left( D_{n}\right) \right\vert =2n$ and $\left\vert E\left( D_{n}\right) \right\vert=3n$. 
Thus, the next result is redily followed from the preceding theorems. This improves the bound given in Theorem \ref{old_bound}.

\begin{theorem}
\label{improved_bound}
For even integers $n$ with $n \geq 4$,
\begin{equation*}
\mu_{s}\left(D_{n}\right) \leq n+1 \text{.}
\end{equation*}
\end{theorem}

It is known from \cite{IO} that $\mu_{s}\left (Q_{3}\right)=5$, where $Q_{3}$ is the $3$-cube. 
Since $D_{4}=Q_{3}$, it follows that $\mu_{s}\left (D_{4}\right)=5$. 
This indicates that the bound given in Theorem \ref{improved_bound} is attained for $n=4$.
However, there is no knowledge whether $\mu_{s}\left(D_{n}\right)=n+1$ for even integers $n$ with $n \geq 6$ so far.
This motivates us to propose the following problem.

\begin{problem}
Determine whether 
\begin{equation*}
\mu_{s}\left(D_{n}\right) = n+1
\end{equation*}
for even integers $n$ with $n \geq 6$. 
\end{problem}

\section{A new approach}

In this section, we propose a new approach to attak Conjecture \ref{super_edge-magic_tree_conjecture}.
For this reason, we now provide the definition for the key concept to be discussed below.

For a graph $G$, a \emph{numbering} $f$ of $G$ is a labeling that assigns distinct elements of the
set $\left[1, \left\vert V\left( G\right) \right\vert \right]$ to the vertices of $G$, where each $uv\in E\left(
G\right) $ is labeled $f\left( u\right) +f\left( v\right) $. 
The \emph{strength} \textrm{str}$_{f}\left( G\right) $ \emph{of a numbering} $%
f:V\left( G\right) \rightarrow \left[1, \left\vert V\left( G\right) \right\vert \right]$ of $G$ is defined by%
\begin{equation*}
\mathrm{str}_{f}\left( G\right) =\max \left\{ f\left( u\right)+f\left(
v\right) \left| uv\in E\left( G\right) \right. \right\} \text{,}
\end{equation*}
that is, $\mathrm{str}_{f}\left( G\right) $ is the maximum edge label of $G$
and the \emph{strength\ }\textrm{str}$\left( G\right) $ of a graph $G$
itself is 
\begin{equation*}
\mathrm{str}\left( G\right) =\min \left\{ \mathrm{str}_{f}\left( G\right)
\left| f\text{ is a numbering of }G\right. \right\} \text{.}
\end{equation*}

This type of numberings was introduced in \cite{IMO} as a generalization of the problem of finding
whether a graph is super edge-magic or not (see Lemma \ref{trivial} or consult \cite{AH} for
alternative and often more useful definitions of the same concept).

There are other related parameters that have been studied in the area of
graph labelings. Excellent sources for more information on this topic are
found in the extensive survey by Gallian \cite{Gallian}, which also includes
information on other kinds of graph labeling problems as well as their
applications.

Several bounds for the strength of a graph have been found in terms of
other parameters defined on graphs (see \cite{GLS, IMT, IMO, IMOT}).
The strengths of familiar classes of graphs were found in \cite{IMO}. 
The strength of trees was also determined by Gao, Lau, and Shiu \cite{GLS} as the next result indicates.

\begin{theorem}
\label{strength}
For every nontrivial tree $T$, 
\begin{equation*}
\mathrm{str}\left( T\right)= \left\vert V\left( T\right) \right\vert+1 \text{.}
\end{equation*}
\end{theorem}

We are now ready to state the following conjecture, which may give us a viable approach towards settling Conjecture \ref{super_edge-magic_tree_conjecture}.

\begin{conjecture}
\label{strength_deficiency_conjecture}
For every nontrivial tree $T$, there exists some positive constant $c$ such that 
\begin{equation*}
\mathrm{str}\left( T\right) \geq c \cdot \mu_{s} \left(T\right)+ \left\vert V\left( T\right) \right\vert+1 \text{.}
\end{equation*}
\end{conjecture}

It is now immediate that if Conjecture \ref{strength_deficiency_conjecture} is true, then Theorem \ref{strength} impies Conjecture \ref{super_edge-magic_tree_conjecture}.

We next consider a graph labeling that is related to super edge-magic labelings. 
Harmonious labelings have been defined and studied by Graham and Sloane \cite{GS} as part of their study of additive bases and are applicable to error-correcting codes. 
A  \emph{harmonious labeling} of a graph $G$ with $\left\vert E\left( G\right) \right\vert \geq \left\vert V\left( G\right) \right\vert$ is an injective function $f:V\left(G\right) \rightarrow \left[1, \left\vert E\left( G\right) \right\vert-1 \right]$ satisfying the condition that the induced edge labeling given by $f\left(u\right) + f\left(v\right) \pmod{\left\vert E\left( G\right) \right\vert }$ for each $uv\in E\left( G\right) $ is also an injective function. 
Furthermore, $G$ is said to be \emph{harmonious} if such a labeling exists. 
This definition extends to trees (for which $\left\vert E\left( G\right) \right\vert = \left\vert V\left( G\right) \right\vert-1$) if at most one vertex label is allowed to be repeated. 

Grace \cite{Grace} introduced sequential graphs, a subclass of harmonious graphs, and showed that any tree admitting an $\alpha$-valuation is sequential and hence is harmonious. 
On the other hand, Lee, Schmeichel, and Shee \cite{LSS} introduced a generalization of harmonious graphs, namely, felicitous graphs.
The following relation among labelings of trees was established in \cite{FIM}.

\begin{theorem}
\label{relation}
If $T$ is a super edge-magic tree, then $T$ is sequential and harmonious.
\end{theorem}

As with super edge-magic labelings, many classes of trees have been shown to be harmonious (see \cite{Gallian} for a detailed list of trees), 
but whether all trees are harmonious is not known.

\begin{conjecture}
\label{harmonious_tree_conjecture}
Every nontrivial tree is harmonious.
\end{conjecture}

Of course, if Conjecture \ref{strength_deficiency_conjecture} is true, so is Conjecture \ref{super_edge-magic_tree_conjecture}.
Ineed, in light of Theorem \ref{relation}, the truth of Conjecture \ref{super_edge-magic_tree_conjecture} in turn 
implies that every nontrivial tree is sequential and the truth of the above conjecture due to Graham and Sloane \cite{GS}
as well as the fact that every nontrivial tree is felicitous.
%as well as every nontrivial tree is felicitous.

\end{document}